\documentclass{article}
\usepackage{amsmath, amsfonts, amssymb, amsthm}
\usepackage{epsfig}
\usepackage{hyperref}
\usepackage{color}

\newlength{\hchng}
\newlength{\vchng}
\newtheorem{thm}{Theorem}[section]
\newtheorem{prop}[thm]{Proposition}
\newtheorem{cor}[thm]{Corollary}

\newtheorem{lemma}[thm]{Lemma}
\newtheorem{defn}[thm]{Definition}

\newtheorem{preremark}[thm]{Remark}
\newenvironment{remark}{\begin{preremark}\rm}{\medskip \end{preremark}}
\numberwithin{equation}{section}

\newcommand{\R}{\mathbb R}
\newcommand{\eps}{\varepsilon}

\newcommand{\grad} {\nabla}
\newcommand{\lap} {\triangle}

\newcommand{\dx} {\; \mathrm{d} x}
\newcommand{\dd} {\; \mathrm{d}}

\newcommand{\subdiff} {\partial}
\newcommand{\MA} {\mathrm{MA}_s}

\title{A nonlocal Monge-Amp\`ere equation}
\author{Luis Caffarelli and Luis Silvestre}

\begin{document}
\maketitle
\begin{abstract}
We introduce a nonlocal analog to the Monge-Amp\`ere operator and show some of its properties. We prove that a global problem involving this operator has $C^{1,1}$ solutions in the full space.
\end{abstract}

\section{Introduction}

The classical Monge-Amp\`ere equation prescribes the values of the determinant of the Hessian of a convex function $u$.
\[ \det D^2 u = f \qquad \text{ in } \Omega.\]

It is an equation with well known applications to differential geometry and mass transportation. It also plays a central role in the regularity theory for elliptic equations in non-divergence form, in part because it can be written as the extreme operator
\[ \left( \det D^2 u(x) \right)^{1/d} = \inf \left\{a_{ij} \partial_{ij} u(x) : \det \{ a_{ij} \} = 1, \{a_{ij}\} > 0 \right\}.\]
This equality holds provided that $D^2 u(x) \geq 0$. Note that when the hessian matrix $D^2 u(x)$ has a negative eigenvalue, the right hand side is $-\infty$.

In this paper we develop a nonlocal version of the Monge-Amp\`ere operator based on this extremal property. It is in some sense a fractional order version of the Monge-Amp\`ere equation. We will see that it satisfies some properties which resemble those of the classical case.

Note that from the expression above, we can deduce that
\[\left( \det D^2 u(x) \right)^{1/d} = \inf \left\{\lap [u \circ A] (x) : \det A = 1 \right\}\]
By $u \circ A$ we mean the composition of the linear transformation $x \mapsto Ax$ with the function $u$. For any $s \in (0,2)$, we could mimic this definition in the fractional order case by
\[ F_s[u](x) = \inf \left\{-(-\lap)^{s/2} [u \circ A] (x) : \det A = 1 \right\}\]
This approach was studied by Luis Caffarelli and Fernando Charro in \cite{caffarelli2014fractional}.

The operator $(-\lap)^{s/2}$ is nonlocal. We want to obtain a Monge-Amp\`ere-like operator which can be understood as an extremal operator of integro-differential type. Because of this, it seems convenient and arguably more natural to take the infimum not over the linear functionals with determinant one, but over all measure preserving transformations. In the second order case, there is no difference between the two approaches, but for  nonlocal equations, the two operators will be significantly different. 

We now describe our definition of the nonlocal Monge-Amp\`ere operator. For every point $x$ in the domain of a convex function $u$, we write $\tilde u$ for the function with its zeroth and first order terms removed in its Taylor expansion
\[ \tilde  u (y) = u(x+y) - u(x) - y \cdot \grad u(x).\]
Then, we can easily check that the Monge-Amp\`ere operator corresponds to
\[ \left( \det D^2 u(x) \right)^{1/d} = \inf \left\{ \lap [\tilde u \circ \varphi](0) : \text{ for all } \varphi \text{ measure preserving s.t. } \varphi(0)=0 \right\}.\]

For any $s \in (0,1)$, we define our fractional order Monge-Amp\`ere-like operator using this philosophy. We write
\[ \MA u(x) = \inf \left\{ -(-\lap)^{s/2} [\tilde u \circ \varphi](0) : \text{ for all } \varphi \text{ measure preserving s.t. } \varphi(0)=0 \right\}.\]

The purpose of this article is to understand this operator and some of its properties. We review and analyze its definition in section \ref{s:definitions}. We also discuss some equivalent formulations. Indeed, the operator can also be defined as the infimum of all integro-differential operators of the form
\[ \int_{\R^d} (u(x+y) - u(x) - y \cdot \grad u(x)) K(y) \dd y,\]
from all kernels $K$ with the same distribution as $c|y|^{-d-s}$.

We also study the solvability and regularity of the following problem in the full space $\R^d$.
\begin{equation} \label{e:global-eq}
\MA u = u-\varphi \qquad \text{in } \R^d.
\end{equation}
Here $\varphi$ is a given strictly convex function which behaves asymptotically as a cone at infinity. This is a similar problem to the one studied in \cite{caffarelli2014fractional}. We prove that there exists a unique solution $u$, which is $C^{1,1}$.

Even though the definition of the operator is very natural, there are several technical difficulties that arise from its study. For example, the H\"older continuity of $\MA u$ when $u \in C^{1,1}$, which is proved in Corollary \ref{c:MA-holder}, or the fact that a supremum of subsolutions is a subsolution, proved in Lemma \ref{l:uissub}, require subtle geometrical considerations.

In section \ref{s:definitions} we give the basic definition of the operator $\MA u$, which is equivalent to the one described above. We explore this definition and extend it to nonsmooth function in section \ref{s:pointwise}. We show that in fact we can make sense of $\MA u(x)$, at every point $x$, with values in $[-\infty,+\infty]$, for any continuous function $u$. In section \ref{s:pointwise} we also show some equivalent reformulations of the operator and we prove that $\MA u(x) \to c (\det D^2 u(x))^{1/d}$ as $s \to 2$. In section \ref{s:properties} we explore several technical properties of the operator $\MA$ related to continuity and convexity. There are some delicate techincal issues involving this operator that are addressed  in this section. In section \ref{s:problem}, we prove the solvability of the equation \ref{e:global-eq}. We finish the article with some comments and open questions related to this new operator.

\section{Definitions}
\label{s:definitions}

\subsection{The nonlocal Monge-Amp\`ere operator}

We study a fractional version of the Monge-Amp\`ere equation. We use the parameter $s$ to represent the order of the equation. In this paper $s$ must be a number in the interval $s \in (1,2)$. We will define the nonlocal Monge-Amp\`ere operator as an infimum of integro-differential operators. Other values of $s$, as well as other families of integral kernels will be the object of future work.

\begin{defn}\label{d:monsterMA}
Given a function $u$ which is $C^2$ at the point $x$, we define $\MA u(x)$ by the following formula
\begin{equation} \label{e:monsterma}  \begin{split}
\MA u(x) = \inf\bigg\{ &\int_{\R^d} (u(x+y) - u(x) - y \cdot \grad u(x)) K(y) \dd y : \\
& \text{ from all kernels $K$ such that } |\{y : K(y)>r^{-d-s}\}| = |B_r| \bigg\}. 
\end{split} 
\end{equation}
\end{defn}

Let us analyze the definition. The infimum is taken over all kernels $K(y)$ whose distribution function coincides with the distribution function of the kernel of the fractional Laplacian of order $s$: $|y|^{-d-s}$.

The operator is naturally well defined when $u$ is a smooth convex function. Conversely, the infimum will be $-\infty$ at some points if $u$ is not convex. This is explained precisely in the next proposition. Observe also that if $u$ has a large growth at infinity (for example if $u(y) \geq c |y|^s$ for all $|y|$ large), then the tail of the integral above would be necessarily divergent. In the next proposition, we assume $u(y) \leq C(|y|+1)^{s-\delta}$ so that the tails of the integral are under control.

\begin{prop} \label{p:hastobeconvex}
Assume $u(y) \leq C(|y|+1)^{s-\delta}$ for some $\delta>0$ and all $y \in \R^d$. If $u$ does not have a supporting tangent plane at $x$, then $\MA u(x) = -\infty$. More precisely, if for some $y \in \R^d$, $u(x+y) < u(x) + y \cdot \grad u (x)$, then $\MA u(x) = -\infty$.
\end{prop}

\begin{proof}
If such value of $y$ exist, we can take as $K$ the kernel of the fractional Laplacian centered at $y$, that is
\[ K(z) = |z-y|^{-d-s}.\]
From the continuity of $u$, $u(z) - u(x) - z \cdot \grad u (x) < 0$ for $z$ in a neighborhood of $y$. Since $K$ is not integrable around $y$, we get
\[ \int_{B_r(y)} (u(x+z) - u(x) - z \cdot \grad u (x)) K(z) \dd z = -\infty.\]
On the other hand, the tail of the integral is an integrable function since $u(y) \leq C(|y|+1)^{s-\delta}$.
\[ z \mapsto (u(x+z) - u(x) - z \cdot \grad u (x)) K(z) \in L^1(\R^d \setminus B_r(y)).\]
Thus we conclude the proof.
\end{proof}

We now review the definition that we gave in the introduction, which is clearly equivalent to the one above. Instead of taking the infimum over kernels $K$ whose distribution coincides with $|y|^{-d-s}$, we can fix this kernel and make a measure preserving change of variables in the function $u$. That is
\begin{equation} \label{e:sup-with-transformation}
 \MA u(x) = \inf \left\{ \int_{\R^d} (u(x+\varphi(y)) - u(x) - \varphi(y) \cdot \grad u(x)) |y|^{-d-s} \dd y \ : \ \varphi \text{ is measure preserving} \right\} 
\end{equation}

Using the standard integral formula for the fractional Laplacian, we can rewrite the above expression as
\begin{equation} \label{e:sup-with-frac-lap}
\begin{aligned}
\MA u(x) = \inf \big\{  -c_{d,s}(-\lap)^{s/2} [\tilde u \circ \varphi ](0) \ : \ \varphi \text{ is measure preserving, } \varphi(0)=0 \text{ and } \\
\tilde u(y) = u(x+y) - u(x) - y \cdot \grad u(x) \big\} 
\end{aligned}
\end{equation}

Here $c_{d,s}$ is a constant depending on dimension and $s$. Its value will only matter when we study the assumptotics for $s \to 2$. In that case, it is important that $c_{d,s} \approx (2-s)$.

\section{Evaluating the operator point-wise on nonsmooth functions}

\label{s:pointwise}

The purpose of this section is to make sense of the operator $\MA u(x)$ pointwise even if $u$ is not $C^2$ at the point $x$.

The definition of supporting plane that we use is the following.
\begin{defn} \label{d:supportingplane}
A supporting plane of $u$ at a point $x$ is a linear function $\ell(y) = a \cdot y + b$ such that $\ell(x) = u(x)$ and $\ell(y) \leq u(y)$ for all $y \in \R^d$.
\end{defn}

We recall that the subdifferential of $u$ at $x$ (which we write $\subdiff u(x)$) is the set of all slopes of the supporting planes. In particular, if $u$ is differentiable at $x$ and convex, the subdifferential is the singleton $\{\grad u(x)\}$. It is elementary to show that in general the subdifferential of a continuous function is a compact convex set at any interior point of its domain. 

It is also convenient to give a separate definition for the subdifferential of $u$ in some set $\Omega$.

\begin{defn} \label{d:bigsubdifferential}
The subdifferential of $u$ in a set $\Omega$ is the set of all vectors $b$ such that the global minimum of $u(x) - b \cdot x$ (for all $x \in \R^d$) is achieved at some point $x \in \Omega$.
\end{defn}

It is a well known fact that the subdifferential in $\Omega$ is the union of all the subdifferentials at all the points $x \in \Omega$ and that the subdifferential is compact and convex if $\Omega$ is compact and convex.

The first proposition is a simple regularity condition which allows us to compute \eqref{e:monsterma} classically and obtain a finite number.

\begin{prop} \label{p:levelsets}
Assume $u(x) \leq C(1+|x|^{s-\delta})$ for some $C>0$ and $\delta>0$. If $u \in C^{1,1}(x)$ and $u$ has a supporting plane at $x$, then $\MA u(x)$ (given in Definition \ref{d:monsterMA}) is a well defined real number.
\end{prop}

\begin{proof} Since $u$ is $C^{1,1}(x)$, in particular $\grad u(x)$ is well defined and the supporting plane at $x$ must have slope $\grad u(x)$. Thus $u(x+y) - u(x) - y \cdot \grad u(x) \geq 0$ for all $y \in \R^d$.

Therefore, for any $K$, the expression
\[ \int_{\R^d} (u(x+y) - u(x) - y \cdot \grad u(x)) K(y) \dd y \geq 0.\]
In order to show that the infimum over all $K$ is a real number we must exhibit one kernel $K$ for which it is finite. Let $K(x) = |x|^{-d-s}$ (the kernel of the fractional Laplacian), then
\[ 
\begin{aligned}
\int_{\R^d} (u(x+y) - u(x) - y \cdot \grad u(x)) K(y) \dd y &\leq
\int_{\R^d} \min\left(C|y|^2,C(|y| + (1+|y|)^{s-\delta}) \right) \frac 1 {|y|^{d+s}} \dd y, \\
&< +\infty \qquad \text{using that $s \in (1,2)$.}
\end{aligned}
\]
\end{proof}

We will now redefine the operator $\MA u(x)$. The following definition will coincide with definition \ref{d:monsterMA} when $u$ is $C^2$ at the point $x$. The advantage of the following definition is that it makes sense for any continuous function $u:\R^d \to \R$ and $\MA u(x) \in [-\infty,+\infty]$.

\begin{defn} \label{d:pointwise}
Given any function $u$, we define $\MA u(x)$ in the following way.
\begin{itemize}
\item[a)] If there is no vector in the subdifferential $\subdiff u(x)$, then $\MA u(x) := -\infty$.
\item[b)] If $\subdiff u(x) \neq \emptyset$, then
\[
 \begin{split}
\MA u(x) = \sup_{b \in \subdiff u} \inf \bigg\{ &\int_{\R^d} (u(x+y) - u(x) - y \cdot b) K(y) \dd y : \\
& \text{ from all kernels $K$ such that } |\{y : K(y)>r^{-n-s}\}| = |B_r| \bigg\}. 
\end{split} 
\]
\end{itemize}
\end{defn}

Note that for any $b \in \subdiff u(x)$, the integrand is nonnegative, so the integral is well defined in the extended nonnegative numbers $[0,+\infty]$.

Part a) of the definition is justified by Proposition \ref{p:hastobeconvex}.

It is important to understand when and how the supremum and infimum in Definition \ref{d:pointwise} are achieved. Let us first fix a vector $b \in \subdiff u$ and analyze the infimum over the admissible kernels $K$.
\begin{equation} \label{e:infoverK} 
 \begin{split}
\inf \bigg\{ &\int_{\R^d} (u(x+y) - u(x) - y \cdot b) K(y) \dd y : \\
& \text{ from all kernels $K$ such that } |\{y : K(y)>r^{-n-s}\}| = |B_r| \bigg\}. 
\end{split} 
\end{equation}
We will see that in most cases it is achieved at a kernel $K$ whose level sets coincide with the level sets of $(u(x+y) - u(x) - y \cdot b)$.

Indeed, since $(u(x+y) - u(x) - y \cdot b)$ is a nonnegative quantity for all $y \in \R^d$, the minimum value of the integral will certainly take place when we locate the maximum values of $K(y)$ coinciding with the smallest values of $(u(x+y) - u(x) - y \cdot b)$. That is, the set $\{K(y) > r^{-d-s}\}$ must coincide with the set of the same measure as $B_r$ where the values of $(u(x+y) - u(x) - y \cdot b)$ are as small as possible. Thus we must find $\lambda >0$ such that
\begin{equation} \label{e:choiceoflevelset}  |\{u(x+y) - u(x) - y \cdot b < \lambda\}| = |B_r|.
\end{equation}
Except for some special functions $u$ which we describe below, for each $r>0$, there is a single $\lambda>0$ for which the equality \eqref{e:choiceoflevelset}, and therefore, there is a unique kernel $K$ which achieves the minimum in \eqref{e:infoverK}.

Two special cases must be addressed in which the above analysis changes. First, it may happen that the measure in the left hand side of \eqref{e:choiceoflevelset} is discontinuous with respect to $\lambda$ and skips the value of $|B_r|$. That is the case when there is a fat level surface,
\begin{align*}
|\{u(x+y) - u(x) - y \cdot b < \lambda\}| &< |B_r|, \\
|\{u(x+y) - u(x) - y \cdot b = \lambda\}| &\geq |B_r|, \\
\intertext{and consequently,}
|\{u(x+y) - u(x) - y \cdot b = \lambda\}| &> 0.
\end{align*}
This is the case in which there is no uniqueness in the choice of $K$ since clearly any measure preserving rearrangement of the values on $K$ within the level set
$\{u(x+y) - u(x) - y \cdot b = \lambda\}$ would not affect the value of the integral. In any case, we must have
\[ \{u(x+y) - u(x) - y \cdot b \leq \lambda\} = \{K \geq r^{-d-s}\},\]
if $r$ is the radius such that
\[ |\{u(x+y) - u(x) - y \cdot b \leq \lambda\}| = |B_r|.\]
In this sense $u(x+y) - u(x) - y \cdot b$ and $K$ still have the same level sets (for each $\lambda$ there is an $r$, but not vice-versa).

The other special case is when $|\{u(x+y) - u(x) - y \cdot b \leq \lambda\}| = +\infty$ for some $\lambda < 0$. If the measure is finite for all $\lambda < \lambda_0$, then we can still find an admissible kernel $K$ supported in $\{u(x+y) - u(x) - y \cdot b \leq \lambda_0\}$ with the same level sets up to the level $\lambda_0$, as described above. In this case $K(y)=0$ for those values of $y$ where $u(x+y) - u(x) - y \cdot b > \lambda_0$.  

The case $\lambda_0 = 0$, is somewhat special. If $|\{u(x+y) - u(x) - y \cdot b = 0\}| = +\infty$, then we can find an admissible kernel $K$ whose support is entirely contained in $\{u(x+y) - u(x) - y \cdot b = 0\}$ and then $\MA u(x) = 0$. The only special case left, and the only case in which the minimum of \eqref{e:infoverK} is not achieved for any admissible $K$, is when $|\{u(x+y) - u(x) - y \cdot b < \lambda\}| = +\infty$ for all $\lambda < 0$, but $|\{u(x+y) - u(x) - y \cdot b = 0\}| < +\infty$. The next lemma says that in this case the expression \eqref{e:infoverK} equals zero.

\begin{lemma} \label{l:caseMs=0}
Assume $|\{u(x+y) - u(x) - y \cdot b < \lambda\}| = +\infty$ for all $\lambda > 0$, then the infimum \eqref{e:infoverK} equals zero.
\end{lemma}

\begin{proof}
We will construct an admissible kernel $K$ which makes the integral arbitrarily small. Note that the measure of the level sets of $K$ are prescribed, and this is our only restriction in the choice of $K$. Thus we construct $K$ by describing where to locate each level set.

Let us define the sets
\begin{align*}
A_{-1} &= \{y : 0 < K(y) \leq 1\},\\
A_k &= \{y : 2^{k(d+s)} < K(y) \leq 2^{(k+1)(d+s)}\},  \qquad \text{for } k \geq 0.
\end{align*}
Note that for $k \geq 0$, $|A_k| = c 2^{-kd}$ for some constant $c$. In particular $\bigcup_{j=0}^\infty |A_k|$ is finite.

For an arbitrary $\eps>0$, we also define the sets of infinite measure
\[ B_k = \{y : u(x+y) - u(x) - y \cdot b < \eps 2^{-ks-k}\}.\] 
We place $A_0$ anywhere inside $B_0$. Since $B_0$ has infinite measure, we can certainly do it. Now we place iteratively $A_k$ inside $B_k \setminus \bigcup_{j=0}^{k-1} A_j$. Since the latter set has infinite measure, we always have enough space to do it.

We finally place $A_{-1}$ inside $B_0 \setminus \bigcup_{j=0}^{k-1} A_j$. Both sets have infinite measure, so again this is allowed.

For this choice of kernel $K$, we obtain
\[
\begin{aligned}
\int_{\R^d} (u(x+y) - u(x) - y \cdot b) K(y) \dd y &\leq
\eps \int_{A_{-1}} K(y) \dd y + \sum_{k=0}^\infty \int_{A_k} \eps 2^{-ks-k} K(y) \dd y,\\
&\leq
C \eps + \sum_{k=0}^\infty C \eps 2^{-k} \leq C \eps.
\end{aligned}
\]
Since $\eps$ is arbitrary, we make this integral arbitrarily small and the infimum in \eqref{e:infoverK} is zero.
\end{proof}

From the analysis above, we see that if the infimum in \eqref{e:infoverK} is positive, then it is achieved by some admissible kernel $K$ whose level sets coincide, in some sense, with the level sets of $(u(x+y) - u(x) - y \cdot b)$. If the infimum is zero, then it may or may not be realized depending on whether the measure of the zero level set $\{u(x+y) - u(x) - y \cdot b=0\}$ is finite or infinite. 

We now turn to the other point of view, given by the definition of $\MA u(x)$ given in \eqref{e:sup-with-transformation} or \eqref{e:sup-with-frac-lap}. It is clear that both definitions are equivalent since, by a change of variables in the integral, we can switch between rearranging the kernel $K$ or the function $\tilde u$. In this case, it is clear that we would achieve the infimum if the function $v = \tilde u \circ \varphi$ is radially symmetric and monotone increasing along the rays. Since $v$ must have the same distribution function as $u$, it must correspond to its radial rearrangement. The radial rearrangement is precisely the unique function $v$ which satisfies the following properties. Recall that $\tilde u(y) = u(x+y) - u(x) - y \cdot p$ for some $p \in \partial u(x)$.
\begin{itemize}
\item $v$ is radially symmetric.
\item $v$ is monotone increasing along rays (i.e. for any $x$, the function $t \mapsto v(tx)$ is monotone increasing for $t \in (0,+\infty)$).
\item For all values of $\lambda>0$, $|\{y: \tilde u(y) < \lambda\}|= |\{y: v(y) < \lambda\}|$.
\end{itemize}

Note that this function $v$ equals $\tilde u \circ \varphi$ for some measure preserving  bijection $\varphi$ provided that $|\{y: \tilde u(y) < \lambda\}| < +\infty$ for all $\lambda$. If $|\{y: \tilde u(y) < \lambda\}| < +\infty$ for all $\lambda < \lambda_0$ but not for $\lambda > \lambda_0$, then $v = \tilde u \circ \varphi$ for a measure preserving function $\varphi$ which is not onto. If $|\{y: \tilde u(y) < \lambda\}| = +\infty$ for all $\lambda>0$ but $|\{y: \tilde u(y) = 0\}| < +\infty$, then $v \equiv 0$, $\MA u(x) = 0$ and the measure preserving transformation $\varphi$ does not exist.

We state the fact already described in the next proposition.

\begin{prop} \label{p:radial-rearrangement}
For every $p \in \partial u(x)$, let $v_p$ be the radial rearrangement of $\tilde u$ described above. Then
\[ \MA u(x) = \sup_{p \in \partial u(x)} -c_{d,s} (-\lap)^{s/2} v_p(0). \]
\end{prop}

\begin{remark}
An important observation is that the function $v_p$ is the same for all values of $s$.
\end{remark}

\begin{prop} \label{p:formula-with-mu}
Assume $u$ is strictly convex at the point $x$. Let
\[ \mu_b(t) := |\{ y : u(y) - u(x) - b \cdot (y-x) < t \}| .\]
Then
\[ \MA u(x) = C \sup_{b \in \subdiff u(x)} \int_0^\infty \frac 1 {\mu_b(t)^{s/d}} \dd t,\]
where $C$ is a constant depending on $d$ and $s$.
\end{prop}

\begin{proof}
Let $v$ be the function from Proposition \ref{p:radial-rearrangement}. Since $v$ is a radial function, we abuse notation by writing $v(|x|) = v(x)$. We have that
\[ \MA u(x) = \int_{\R^d} v(y) \frac 1 {|y|^{d+s}} \dd y = d |B_1| \int_0^\infty \frac{v(r)} {r^{1+s}} \dd r.\]
We observe that from the definition of $v$, $\mu(v(r)) = |B_1| r^d$. Moreover, $\mu(v(r))$ is a monotone increasing function of $r$. Differentiating this identity, we get,
\[ \dd \mu(v(r)) = d |B_1| r^{d-1} \dd r.\]

Replacing in the formula for $\MA u(x)$, we obtain
\[ \MA u(x) = \int_0^\infty \frac{v(r)}{\mu(v(r))^{1+s/d}} \dd \mu(v(r)).\]
Making the change of variables $t = v(r)$ in the integral above, we get
\[ \MA u(x) = \int_0^\infty \frac{t}{\mu(t)^{1+s/d}} \dd \mu(t) = \frac{d} s \int_0^\infty \frac{1}{\mu(t)^{s/d}} \dd t.\]
The last equality follows from integration by parts.
\end{proof}

\subsection{Limit as $s \to 2$}

The following proposition explains how the classical Monge-Amp\`ere operator is the limit of $\MA$ as $s \to 2$.

\begin{prop} \label{p:limitofoperator}
Let $u$ be a $C^2$ function which has a supporting plane at $x$. Assume that $u(x) \leq C|x|^{2-\sigma}$ for some $\sigma > 0$. Then 
\[\lim_{s \to 2} (2-s) \MA u(x) = c_d(\det D^2u(x))^{1/d}\]
 for some constant $c_d$ depending on dimension only.
\end{prop}

\begin{proof} 
We use Proposition \ref{p:radial-rearrangement}. Indeed, since we assume in this case that $u$ is differentiable at $x$, then $\partial u$ has only one point which equals $\grad u(x)$. Let $v(y)$ be the radial rearrangement of $u(x+y) - u(x) - y \cdot \grad u(x)$. We have that

\[ \MA u(x) = -c_{d,s} (-\lap)^{s/2} v(0).\]

Using that $(2-s) c_{d,s} \to c_d$ as $s \to 2$ and $(-\lap)^{s/2} v(0) \to (-\lap) v(0)$ as $s \to 2$, we get
\[ \lim_{s \to 2} (2-s) \MA u(x) = c_d \lap v(0) = c_d [ \det D^2 u(x) ]^{1/d}.\]
\end{proof}


\section{Some useful properties of $\MA$}

\label{s:properties}

In this section we prove a lemma on the monotonicity of the operator $\MA$, a lemma on the concavity of $\MA$, and a lemma on the lower semicontinuity of $\MA u$ for $u$ convex.

\begin{lemma}[Monotonicity of $\MA$] \label{l:monotonicity}
Let $u$ and $v$ be two functions such that $u(x_0) = v(x_0)$ and $u(x) \geq v(x)$ for all $x \in \R^d$. Then 
$\MA u(x_0) \geq \MA v(x_0)$.
\end{lemma}

\begin{proof}
We first observe that all tangent planes of $v$ from below at $x_0$ are also tangent planes of $u$ at the same point. Therefore, $\subdiff v(x_0) \subset \subdiff u(x_0)$. Since $\MA u(x_0)$ is defined as a supremum over all $b \in \subdiff u(x_0)$, we obtain a lower bound if we restrict this suppremum over the elements $b \in \subdiff v(x_0)$.

For every $b \in \subdiff v(x_0)$ and any nonnegative kernel $K$, we have 
\[ \int_{\R^d} (u(x_0+y) - u(x_0) - y \cdot b) K(y) \dd y \geq \int_{\R^d} (v(x_0+y) - v(x_0) - y \cdot b) K(y) \dd y.\]
Taking infimum over all the admissible kernels $K$ and suppremum over $b \in \subdiff v(x_0)$ we get.
\[ \begin{split}
\MA u(x_0) &\geq \sup_{b \in \subdiff v(x_0)} \inf_K \int_{\R^d} (u(x_0+y) - u(x_0) - y \cdot b) K(y) \dd y \\
&\geq \sup_{b \in \subdiff v(x_0)} \inf_K \int_{\R^d} (v(x_0+y) - v(x_0) - y \cdot b) K(y) \dd y = \MA v(x_0).
\end{split}\]
\end{proof}


\begin{lemma}[concavity of $\MA$] \label{l:concavity}
Let $u$ and $v$ be two continuous functions, then
\[ \MA \left( \frac{u+v}2 \right) (x) \geq \frac{\MA u(x) + \MA v(x)} 2 . \] 
\end{lemma}

\begin{proof}
If either $\MA u(x) = -\infty$ or $\MA v(x) = -\infty$ there is nothing to prove. So, let us assume neither thing happens.

Let $b_1 \in \subdiff u(x)$ and $b_2 \in \subdiff v(x)$. Then, we see that $(b_1 + b_2)/2 \in \subdiff (u+v)/2$. An infimum of linear operators is concave, so we have
\[ \begin{split}
\frac{\inf_K \{ \int (u(x+y) - u(x) - b_1 \cdot y) K(y) \dd y + \inf_K \{ \int (v(x+y) - v(x) - b_2 \cdot y) K(y) \dd y} 2 \\ \leq \inf_K \{ \int \left(\frac{u+v}2 (x+y) - \frac{u+v}2(x) - \frac{b_1+b_2}2 \cdot y \right) K(y) \dd y 
\end{split}
\]
Taking supremum in $b_1$ and $b_2$ we obtain $\MA u(x)+\MA v(x)$ in the left hand side. In the right hand side, the set of vectors $(b_1+b_2)/2$ form a subset of $\subdiff (u+v)/2$, therefore we get a lower bound for $\MA (u+v)/2$. Therefore, we obtain
\[ \frac{\MA u(x) + \MA v(x)} 2 \leq \MA \left( \frac{u+v}2 \right).\]
This finishes the proof.
\end{proof}

\begin{lemma}[lower semicontinuity of $\MA u$] \label{l:lower-semicontinuous}
Let $u$ be convex and $C^1$ (with uniform modulus of continuity in $\R^d$), then the function $\MA u$ is lower semicontinuous.
\end{lemma}

\begin{proof}
We will construct an increasing sequence of operators $\MA^n$  such that $\MA^n u$ is continuous for all $n$ and $\MA^n u \to \MA u$ pointwise. Since the limits of increasing sequences of continuous functions are always semicontinuous, this proves the lemma.

In order to construct $\MA^n$, we use Definition \ref{d:pointwise} but with a modified family of truncated kernels. That is, we replace every kernel $K$ in Definition \ref{d:pointwise}, with $K^n(y) = \min(K(y),n)$. This new kernel is in $L^1$ and therefore the corresponding operator
\[ L^n u(x) = \int (u(x+y) - u(x) - y \cdot \grad u(x)) K^n(y) \dd y, \]
is continuous, with a modulus of continuity depending on $\|K^n\|_{L^1}$ and the modulus of continuity of $u$ and $\grad u$. Therefore, for every admissible kernel $K$ in Definition \ref{d:pointwise}, the function $L^n u$ have a modulus of continuity (uniform in the choice of $K$) and 
\[ \MA^n u(x) = \inf L^n u(x)\]
is also continuous. Clearly, $\MA u(x) = \lim_{n \to \infty} \MA^n u(x)$, thus $\MA u(x)$ is lower semicontinuous.
\end{proof}

\begin{remark}
Without assuming $u \in C^1$, the function $\MA u$ may be neither lower nor upper semicontinuous. For example, if $u(x) = |x|$, then $\MA u(0) = +\infty$ whereas $\MA u (x)=0$ for all $x \neq 0$. If $u$ is a polygonal function with vertices $(1/k,1/k^2)$ for all nonzero integers $k$, then $\MA u(0) < +\infty$ whereas $\MA u(1/k) = +\infty$ for all nonzero integers $k$.
\end{remark}

Our next objective is to show that when $u$ is $C^{1,1}$ and strictly convex, then $\MA u$ is H\"older continuous $C^{1-s/2}$. For this, we define the sections
\[ D_x u(t) := \{ y : u(y) - u(x) - (y-x) \cdot \grad u(x) \leq t\}.\]
The definition makes sense provided that $u \in C^1$ so that $\grad u(x)$ is well defined (otherwise there would be some extra ambiguity in the choice of some element in $\partial u(x)$).

\begin{lemma} \label{l:sections}
Assume $u \in C^1$ and convex and $\Lambda = \mathrm{diam} \ D_{x_0} u(t) < +\infty$. Let $x_1$ be such that $2\Lambda |\grad u(x_1)-\grad u(x_0)| < t$ and $|x_0-x_1| \leq \Lambda$. Then
\[ D_{x_0} u\left(t - 2 \Lambda  |\grad u(x_1)-\grad u(x_0)|\right) \subset D_{x_1} u(t)  .\]
\end{lemma}

\begin{proof}
Let $y \in D_{x_0}\left(t - 2 \Lambda  |\grad u(x_1)-\grad u(x_0)|\right)$. We estimate
\begin{align*}
u(y) - u(x_1) - (y-x_1) \cdot \grad u(x_1) &= \left( u(y) - u(x_0) - (y-x_0) \cdot \grad u(x_0) \right)  \\
& \phantom{= \ } - \left( u(x_1) - u(x_0) - (x_1-x_0) \cdot \grad u(x_0) \right)\\
& \phantom{= \ } + (x_0-y) \cdot (\grad u(x_0) - \grad u(x_1)) \\
& \leq \left(t -  2 \Lambda  |\grad u(x_0)-\grad u(x_1)|\right) - 0 +  2 \Lambda  |\grad u(x_0)-\grad u(x_1)| \leq t
\end{align*}
\end{proof}

\begin{lemma} \label{l:MA-pre-holder}
Assume $u$ is convex and $D_{x_0}u(\eps)$ has diameter $\Lambda < +\infty$. Then, for any $x_1$ such that $2\Lambda |\grad u(x_1)-\grad u(x_0)| < \eps/2$ and $|x_0-x_1| \leq \Lambda$, then
\[ \MA u(x_1) - \MA u(x_0) \leq C |\grad u(x_1)-\grad u(x_0)|^{1-s/2} + {\frac {4\Lambda} \eps |\grad u(x_0) - \grad u(x_1)|} \MA u(x_0)\]
\end{lemma}

Note that we will apply Lemma \ref{l:MA-pre-holder} to functions for which $\MA u$ is bounded. Therefore, the first term in the right hand side of the estimate is the most significant.

\begin{proof}
From Lemma \ref{l:sections}, $D_{x_1}u(\eps/2) \subset D_{x_0}u(\eps)$ for $x_1$.

We estimate $\MA u(x_1) - \MA u(x_0)$ using Proposition \ref{p:formula-with-mu}. We have
\[ \MA u(x_1) = C \int_0^\infty \frac 1 {\mu(t)^{s/d}} \dd t, \]
where
\[ \mu(t) = |D_{x_1}u(t)|.\]
Since $u \in C^{1,1}$, $\mu(t) \geq c t^{d/2}$. Therefore
\begin{equation} \label{e:patch2}
 \int_0^{2 \Lambda|\grad u(x_1)-\grad u(x_0)|} \frac 1 {\mu(t)^{s/d}} \dd t \leq C |\grad u(x_1)-\grad u(x_0)|^{1-s/2}.
\end{equation}

Here $C$ depends on $s$ and $[u]_{C^{1,1}}$.

Now we estimate the integral in the range $t \in [2 \Lambda|\grad u(x_1)-\grad u(x_0)|, \eps]$. We use Lemma \ref{l:sections} for this part.
\begin{equation} \label{e:patch1}
 \int_{2 \Lambda|\grad u(x_1)-\grad u(x_0)|}^\eps \frac 1 {|D_{x_1} u(t)|^{s/d}} \dd t \leq  \int_{0}^{\eps-2 \Lambda|\grad u(x_1)-\grad u(x_0)|} \frac 1 {|D_{x_0} u(t)|^{s/d}} \dd t .
\end{equation}

In order to estimate the remaining part of the integral, we need an estimate for $\mathrm{diam} \ D_{x_0} u(t)$ for $t > \eps$. We obtain this estimate using convexity. Let $y \in D_{x_0} u(t)$. There is a $z \in \partial D_{x_0} u(\eps)$ such that $y-x_0 = \lambda (z-x_0)$ for some $\lambda > 1$ ($z$ is the intersection of $\partial D_{x_0}u(\eps)$ with the line segment with endpoints at $x_0$ and $z$). From the definition of convexity,
\[ u(z) \leq \frac{\lambda-1}\lambda u(x_0) + \frac{1}\lambda u(y).\]
Consequently,
\begin{align*}
t = u(y) - u(x_0) - (y-x_0) \cdot \grad u(x_0) \geq \lambda \big(u(z) - u(x_0)- (z-x_0) \cdot \grad u(x_0) \big) = \lambda \eps.
\end{align*}
This means that $\lambda \leq t/\eps$. In other words, the set $D_{x_0} u(t)$ is contained in the homothety of $D_{x_0} u(\eps)$ centered at $x_0$ with ratio $t/\eps$. In particular
\[ \mathrm{diam} \ D_{x_0} u(t) \leq \frac t \eps \ \mathrm{diam} \ D_{x_0} u(\eps) = \frac t \eps \Lambda.\]

Therefore
\begin{align*}
\int_\eps^\infty \frac 1 {|D_{x_1} u(t)|^{d/s}} \dd t \leq \int_\eps^\infty \frac 1 {|D_{x_0} u(t - 2 \Lambda t / \eps |\grad u(x_0) - \grad u(x_1)|)|^{d/s}} \dd t, \\
= \frac 1 {1-\frac {2\Lambda} \eps |\grad u(x_0) - \grad u(x_1)|} \int_{(\eps - 2\Lambda |\grad u(x_0) - \grad u(x_1)|)} ^ \infty \frac 1 {|D_{x_0} u(t)|^{d/s}} \dd t
\end{align*}

Adding up the term in \eqref{e:patch1},
\begin{align*} 
 \int_{2 \Lambda|\grad u(x_1)-\grad u(x_0)|}^\infty 
\frac 1 {|D_{x_1} u(t)|^{d/s}} \dd t &\leq
\frac 1 {1-\frac {2\Lambda} \eps |\grad u(x_0) - \grad u(x_1)|} \int_{0} ^ \infty \frac 1 {|D_{x_0} u(t)|^{d/s}} \dd t , \\ 
 &= \frac 1 {1-\frac {2\Lambda} \eps |\grad u(x_0) - \grad u(x_1)|} \MA u(x_0).
\end{align*}

Combining with \eqref{e:patch2},
\[ \MA u(x_1) \leq  \frac 1 {1-\frac {2\Lambda} \eps |\grad u(x_0) - \grad u(x_1)|} \MA u(x_0) + C |\grad u(x_1) - \grad u (x_0)|^{1-s/2}.\]

Therefore
\begin{align*} 
 \MA u(x_1) - \MA u(x_0) &\leq \frac {\frac {2\Lambda} \eps |\grad u(x_0) - \grad u(x_1)|} {1-\frac {2\Lambda} \eps |\grad u(x_0) - \grad u(x_1)|} \MA u(x_0) + C |\grad u(x_1) - \grad u (x_0)|^{1-s/2} , \\ 
 &\leq {\frac {4\Lambda} \eps |\grad u(x_0) - \grad u(x_1)|} \MA u(x_0) + C |\grad u(x_1) - \grad u (x_0)|^{1-s/2}.
\end{align*}

\end{proof}

\begin{cor} \label{c:MA-holder}
Assume $u$ is convex and $C^{1,1}$, $\MA u$ is bounded in a neighborhood of $x_0$, and $D_{x_0}u(\eps)$ has diameter $\Lambda < +\infty$. Then $\MA u \in C^{1-s/2}$ in a neighborhood of $x_0$.
\end{cor}

\begin{proof}
This follows simply by applying Lemma \ref{l:MA-pre-holder} plus the fact that since $u \in C^{1,1}$, $|\grad u(x_0) - \grad u(x_1)| \leq C |x_1-x_0|$.
\end{proof}





\section{A concrete problem}

\label{s:problem}

We consider the following problem in the full space $\R^d$. Given a smooth, strictly convex function $\varphi :\R^d \to \R$, we look for a function $u$ which solves the equation
\begin{equation} \label{e:problem} 
\MA u = u-\varphi \qquad \text{in } \R^d.
\end{equation}
We will also assume that $\varphi$ behaves asymptotically as a cone at infinity. More precisely
\[ \lim_{R \to \infty} \frac{\varphi(Rx)}R = \Phi(x),\]
where $\Phi$ is a homogeneous function of degree one which is smooth away from the origin.

We will prove that the equation \eqref{e:problem} has a unique solution $u$ which converges to $\varphi$ at infinity. Moreover, this solution is $C^{1,1}$. This is one the the simplest model problems for $\MA$ and we use it as an example of the solvability of an equation involving $\MA$.

We start with the comparison principle for this equation.

\begin{prop} \label{p:comparison}
Let $u$ and $v$ be two continuous functions in $\R^d$ satisfying, for some $\Omega \subset \R^d$ open and bounded,
\begin{align*}
\MA u - u &\geq f \geq \MA v - v \qquad \text{(pointwise)}\\
u &\leq v \qquad \text{in } \R^d \setminus \Omega
\end{align*}
Then $u \leq v$ in $\R^d$
\end{prop}

\begin{proof} Assume the contrary. Then $u-v$ takes a positive absolute maximum at some point $x_0 \in \Omega$. From Lemma \ref{l:monotonicity} (applied to $u$ and $v+u(x_0)-v(x_0)$), we know that $\MA v(x_0) \geq \MA u(x_0)$. But this contradics the equation since we get
\[ \MA v(x_0) - v(x_0) < \MA u(x_0) + u(x_0).\]
\end{proof}

In order to prove the existence of a solution, we will carry out a Perron's method approach. That is, we will consider the maximum of all sub-solutions and prove it is a solution. The first step in this method is to identify a particular sub-solution and a particular super-solution with the right behaviour at infinity. That is the purpose of the following two lemmas.

\begin{lemma} [The lower barrier]
The function $\varphi$ is a sub-solution of \eqref{e:problem}.
\end{lemma}

\begin{proof}
Since $\varphi$ is convex, then $\MA \varphi \geq 0$. Thus, $\MA \varphi \geq \varphi - \varphi$ and $\varphi$ is a sub-solution.
\end{proof}

\begin{lemma} [The upper barrier]
There exists a function $w$ such that $w(x) \leq C (1+|x|)^{1-s}$ for some constant $C$ and $\varphi+w$ is a super-solution of \eqref{e:problem}.

Note that $1-s < 0$ since we are under the hypothesis that $s \in (1,2)$.
\end{lemma}

\begin{proof}
The key observation for this lemma is that $-(-\lap)^{s/2} v(x) \geq \MA v(x)$ for any function $v$. This is simply because the fractional Laplacian corresponds to the kernel $K(y) = |y|^{-d-s}$ in \eqref{e:monsterma}. 

Let $w$ be the function $w := -(I + (-\lap)^{s/2})^{-1} (-\lap)^{s/2} \varphi$. Note that for this function $w$, 
\[ \MA (\varphi + w) \leq -(-\Delta)^{s/2} (\varphi + w) = -(-\lap)^{s/2} \varphi + ((-\lap)^{s/2} \varphi + w) = (\varphi+w) - \varphi.\]
Thus, $\varphi+w$ is indeed a supersolution. We are left to justify that $w$ satisfies the right bounds.

Since $\varphi$ is smooth, $(-\lap)^{s/2} \varphi$ is bounded. Moreover, since $\varphi$ is asymptotically homogeneous of degree one at infinity, we have from homogeneity that $(-\lap)^{s/2} \varphi(x) \approx C |x|^{1-s}$ for large values of $|x|$.

The operator $(I+(-\lap)^{s/2})^{-1}$ is a convolution operator with a kernel $K$ such that $|K(y)| \approx |y|^{-d+s}$ for small values of $|y|$ and $|K(y)| \approx |y|^{-d-s}$ for large values of $|y|$. In particular, $K$ is integrable and it is not hard to show that $K \ast (-\lap)^{s/2} \varphi$ is bounded and decays like $|y|^{1-s}$ for large values of $|y|$.
\end{proof}

The previous two lemmas provide us with a lower bound $\varphi$ and an upper bound $\varphi+w$ for the solution $u$ to \eqref{e:problem}. We will construct $u$ as the maximum of all sub-solutions $v$ of \eqref{e:problem} so that $\varphi+w \geq v \geq \varphi$. This is a nonempty class of functions, since $v=\varphi$ is admissible. Moreover, all these sub-solutions are convex from Proposition \ref{p:hastobeconvex}. In particular, all these sub-solutions $v$ are uniformly Lipschitz. Thus, we define
\begin{equation} \label{e:Perron}
 u(x) = \sup \{ v(x): \varphi \leq v \leq \varphi+w \text{ and } \MA v \geq v-\varphi \}. 
\end{equation}

A supremum of convex functions is also convex, therefore $u$ is a convex function such that $\varphi \leq u \leq \varphi+w$. In particular, $u$ is also Lipschitz.

Before proving that $u$ is indeed a solution of \eqref{e:problem}, we will prove a few simple properties of $u$.

\begin{lemma} \label{l:bounde-flat-parts}
Let $\ell$ be any affine function. Then the set $\{ u = \ell\}$ is compact.
\end{lemma}

\begin{proof}
Since the set is clearly closed, we must prove it is bounded. Assume it was not. Since $\varphi \leq u \leq \varphi + w$ and both the lower and upper bound converge to the cone $\Phi$ at infinity, then necessarily $\ell$ must be one of the supporting planes of $\Phi$. However, that implies that $\ell \leq \Phi < \varphi$ because of the strict convexity of $\varphi$, which contradicts that $\ell = u \geq \varphi$ at some point.
\end{proof}

\begin{lemma} \label{l:c11}
The function $u$ is $C^{1,1}(\R^d)$ and
\[ \|u\|_{C^{1,1}} \leq \|\varphi\|_{C^{1,1}}. \]
\end{lemma}

\begin{proof}
As we pointed out above, the function $u$ is convex, which means that $D^2 u \geq 0$. We are left to prove that $D^2 u \leq \|\varphi\|_{C^{1,1}}$.

Since $\varphi$ is smooth, let $K = \max |D^2 \varphi|$. Thus, for any $x,y \in \R^d$, we have
\[ \varphi(x+y) + \varphi(x-y) - K |y|^2 \leq 2\varphi(x).\]

Let $v$ be any sub-solution of \eqref{e:problem} so that $\varphi \leq v \leq \varphi+w$. That is, $v$ is any of the admissible subsolutions used in Perron's method to obtain $u$ as its minimum.

From Lemma \ref{l:concavity} and the homogeneity of $\MA$ we obtain the following (all operators are applied with respect to $x$, with $y$ fixed),
\[ 
\begin{aligned}
\MA\bigg[ & \frac{v(x+y)+v(x-y) - K|y|^2} 2  \bigg] 
\geq \frac{\MA[v(x+y)] + \MA[v(x-y)]} 2 \\
&\qquad = \frac{v(x+y) - \varphi(x+y) + v(x-y) - \varphi(x-y)} 2 \\
&\qquad= [ \frac{v(x+y)+v(x-y) - K|y|^2} 2 ] - [ \frac{\varphi(x+y)+\varphi(x-y) - K|y|^2} 2 ] \\
&\qquad \geq \frac{v(x+y)+v(x-y) - K|y|^2} 2 - \varphi(x).
\end{aligned}
\]
This means that $[ v(x+y)+v(x-y) - K|y|^2 ]/2$ is a subsolution of \eqref{e:problem}. Moreover, for each fixed $y$,
\[ \begin{aligned}
\sup_{|x|>r} \frac{v(x+y)+v(x-y) - K|y|^2} 2 & - (\varphi(x)+w(x)) \leq \\
& \leq \frac{ \varphi(x+y) + w(x+y)} 2 + \frac{ \varphi(x-y) + w(x-y) } 2 - \frac K 2 |y|^2 - \varphi(x) \\
&\leq \frac{w(x+y) + w(x-y)} 2 \leq C(y) |x|^{1-s}.
\end{aligned}
\]

In the last inequality we used that $0 \leq w(x) \leq |x|^{1-s}$. Therefore, for any $\eps>0$, $(v(x+y)+v(x-y) - K|y|^2)/2 - \eps < \varphi(x)$ for sufficiently large $x$. So, we construct the function $\tilde v(x) = \max\left((v(x+y)+v(x-y) - K|y|^2)/2 - \eps, \varphi(x) \right)$. This function $\tilde v$ is a subsolution of \eqref{e:problem} which coincides with $\varphi$ for large enough $|x|$. From the comparison principle (Proposition \ref{p:comparison}), $\tilde v \leq \varphi + w$. Thus, $v$ is another admissible subsolution for Perron's method. Since $u$ is the supremum of all such subsolutions, then $u(x) \geq \tilde v(x)$. But this is true for all values of $\eps>0$, therefore, for any admissible subsolution $v$,
\[ u(x) \geq \frac{v(x+y)+v(x-y)-K|y|^2} 2.\]
Recall that $u$ is the supremum of all admissible subsolutions $v$. Therefore, for every fixed $x$ and $y$ and an arbitrary $\eps>0$, there is an admissible subsolution $v_1$ so that $u(x+y) - \eps < v_1(x+y) \leq u(x+y)$. Likewise, there is a $v_2$ so that $u(x-y) - \eps < v_2(x-y) \leq u(x-y)$. If we take $v = \max(v_1,v_2)$, we obtain an admissible subsolution $v$ for which both $v(x+y)$ and $v(x-y)$ are larger than $u(x+y)-\eps$ and $u(x-y)-\eps$ respectively. Therefore, taking $\eps \to 0$,
\[ u(x) \geq \frac{u(x+y)+u(x-y) - K|y|^2} 2 \]
for any $x,y \in \R^d$. This clearly implies that $D^2 u \leq K$, from which the $C^{1,1}$ regularity follows.
\end{proof}

\begin{cor}
The function $\MA u(x)$ is locally bounded.
\end{cor}

\begin{proof}
Since $u$ is $C^{1,1}$ and is asymptotically a cone at infinity, then $\MA u(x)$ is a well defined real number and is locally bounded.
\end{proof}

In order to show that $u$ is the solution that we are looking for, we will have to show separately that $\MA u \geq u-\varphi$ and $\MA u \leq u-\varphi$.

\begin{lemma}[$u$ is a subsolution] \label{l:uissub}
The function $u$ is a sub-solution of \eqref{e:problem}
\[ \MA u \geq u-\varphi \]
\end{lemma}

We give two different proofs of this lemma. The proofs are based on very different arguements, so we believe it is interesting to include both points of view.

\begin{proof}[Geometric proof of Lemma \ref{l:uissub}.]
From Lemma \ref{l:bounde-flat-parts}, $u$ cannot coincide with its supporting plane at $x$ in an unbounded set. Consequently, for any $\eps > 0$, the set $D_x u(\eps)$ is bounded. Let $\Lambda := \mathrm{diam} \ D_x u(\eps)$. Note that by convexity, $|\grad u(z) - \grad u(x)| \geq \eps/\Lambda$ for any $z \notin D_x u(\eps)$.

Let $x$ be any point. From Lemma \ref{l:c11}, we know $u \in C^{1,1}$ and therefore there is a quadratic polynomial $P$ which touches $u$ from above at the point $x$. Let us choose $D^2 P = 2[u]_{C^{1,1}} I$.

From our construction, $u$ is the suppremum of all admissible subsolutions $v$ as in \eqref{e:Perron}. Therefore, for an arbitrarily small positive number $h$, we can find a subsolution $v$ such that $P-h^2$ touches $v$ from above at a point $y$. We observe that $|x-y| \leq Ch$ and $|\grad P(y) - \grad P(x)| \leq 2h$, where $C$ is a constant depending on $[u]_{C^{1,1}}$. 

Let $\ell$ be the affine function tangent to $P$ at the point $y$. Since $v$ is convex and it is tangent from below to the quadratic polynomial $P$, then $v$ will have $\ell$ as its only supporting plane at $y$.

Since $\ell$ is a supporting plane of $v$ at $y$ and $v \leq u$, then $\ell \leq u$. Let $k \geq 0$ so that $\ell + k$ is a supporting plane of $u$ at some point $z$. Since $v(y) + h^2 = P(y) \geq u(y)$, then $k \leq h^2$.

Note that $p := \grad u(z) = \grad \ell = \grad P(y) = \grad v(y)$. In particular $|\grad u(z) - \grad u(x)| \leq 2h$. Then, $z \in D_x u(\eps)$ provided that $2h < \eps / \Lambda$. Consequently, $|x-z| \leq \Lambda$. Using Lemma \ref{l:MA-pre-holder}, 
\[ \MA u(z)  \leq  \MA u(x) + C h^{1-s/2}.\]

The slices of $v$, starting at the level $h^2$, contain the slices of $u$. More precisely, since $u \geq v$ and $\ell$ is the supporting plane of $v$ at $y$ and $\ell+k$ is the supporting plane of $u$ at $z$, 
\[ D_z u (t) \subset D_y v(t+k).\]

From proposition \ref{p:formula-with-mu} and the definition of $D_zu(t)$, 
\[ \MA u(z) = c \int_0^\infty \frac 1 {|D_z u(t)|^{s/d}} \dd t \geq c \int_k^\infty \frac 1 {|D_y v(t)|^{s/d}} \dd t = \MA v(y) - c\int_0^k \frac 1 {|D_y v(t)|^{s/d}} \dd t. \]

We estimate the last term using that the polynomial $P$ touches $v$ from above at the point $y$, therefore
\[ \int_0^k \frac 1 {|D_y v(t)|^{s/d}} \dd t \leq C \int_0^k \frac 1 {t^{s/2}} \dd t = C k^{1-s/2}.\]
Thus, 
\begin{align*}
\MA u(x) + C h^{1-2s} &\geq \MA u(z) \geq \MA v(y) - C k^{1-s/2}, \\
&\geq v(y) - \varphi(y)-C h^{1-s/2}, \\
&\geq u(y) -h^2- \varphi(y)-C h^{1-s/2}.
\end{align*}
Since $|x-y| < Ch$, sending $h \to 0$ and using the continuity of $u$ and $\varphi$, we get
\[ \MA u(x) \geq u(x) - \varphi(x).\]
\end{proof}

\begin{proof}[Proof of Lemma \ref{l:uissub} by duality.]
The function $u$ is the supremum of the family of convex functions $v$, which are sub-solutions, uniformly Lipschitz, and are trapped between $\varphi$ and $\varphi+w$. From Lemma \ref{l:monotonicity}, we see that the maximum of any two admissible subsolutions $v$ is also an admissible subsolution. Recall that $w(x) \approx |x|^{1-s}$ for $|x|$ large and in particular $w \to 0$ at infinity. From the Arzela-Ascoli theorem, the set of admissible subsolutions is a compact family of functions with respect to uniform convergence. By a standard diagonalization method, we can construct a sequence of admisible sub-solutions $v_n$ such that $v_n \to u$ uniformly in $\R^d$. Moreover, since $\varphi \leq v_n \leq \varphi+w$, then also $\int v-v_n \dx \to 0$ by the dominated convergence theorem.

Let $r>0$ be an arbitrary, small, positive number. Let $L$ be a linear integro-differential operator corresponding to some admissible kernel $K$ in the definition of $\MA$. Assume moreover that we have $K(y) = |y|^{-d-s}$ every time $|y|<r$. In particular, this implies $K(y)< r^{-d-s}$ for $|y| > r$.

From the definition of $\MA$, we have that $Lv_n \geq \MA v_n \geq v_n - \varphi$. The operator $L$ is a classical linear, translation invariant, integro-differential operator, therefore the inequality $L v_n \geq v_n - \varphi$ can be understood in the sense of distributions. That is, for any smooth, nonnegative, compactly supported, test function $\eta$, 
\[ \int v_n (L^\ast \eta) \dx \geq \int (v_n-\varphi) \eta \dx.\]
Here the operator $L^\ast$ (the adjoint of $L$) is the linear integro-differential operator corresponding to the kernel $K^\ast(y) = K(-y)$. Since $\eta$ is smooth and compactly supported, $L^\ast \eta$ is a smooth function in $L^1(\R^d)$. Thus, we can use the dominated convergence theorem (using that $u-v_n \leq w$ uniformly) to pass to the limit $v_n \to u$ and obtain
\[ \int u (L^\ast \eta) \dx \geq \int (u-\varphi) \eta \dx.\]
Thus, $Lu \geq u -\varphi$ for all these choices of $L$. Note that from Lemma \ref{l:c11}, $u$ is $C^{1,1}$ and therefore $Lu$ is a continuous (in fact $C^{2-s}$, this follows easily for example from Proposition 2.6 in \cite{silvestre2007regularity}) function and the inequality holds classically.

In order to conclude $\MA u \geq u - \varphi$, we would need to know the above inequality for any admissible kernel $K$ (as if we chose $r=0$). So, let $K$ be any admissible kernel in Definition \ref{d:pointwise} of $\MA u$.

For any small value of $r>0$, we construct an approximated kernel $K_r(y)$ in the following way.
\[
K_r(y) = \begin{cases}
|y|^{-n-s} & \text{ if } |y|<r, \\
K(y) & \text{ if } |y|>r \text{ and } K(y) < r^{-d-s}.
\end{cases}
\]
For all the points $y$ which do not fit the criteria above, we define $K_r(y)$ to be any rearrangement of the values of $K(y)$ in $B_r $ which fall in the interval $(0, r^{-d-s})$. Let $L_r$ be the linear integro-differential operator corresponding to $K_r$. From the argument above, we know that
\[ L_r u \geq u - \varphi.\]
Now we will estimate $Lu - L_r u$ from below. Recall that since $u$ is convex, all incremental quotients are nonnegative.
\begin{align*}
Lu - L_r u &= \int_{\R^d} (u(x+y) - u(x) - y \cdot \grad u(x)) (K(y)-K_r(y)) \dd y, \\
&\geq - \int_{B_r} (u(x+y) - u(x) - y \cdot \grad u(x)) |y|^{-d-s} \dd y,\\
&\geq -C r^{2-s}.
\end{align*}
The fact that $u \in C^{1,1}$ was used in the last inequality. Therefore we obtain
\[ Lu \geq u - \varphi - C r^{2-s}.\]
Since $r$ is arbitrarily, we conclude the proof by making $r \to 0$.
\end{proof}


We continue to prove that $u$ is a supersolution of \eqref{e:problem}, that is $\MA u \leq u-\varphi$.

\begin{lemma}[$u$ is a supersolution] \label{l:uissuper}
Let $b \in \subdiff u(x)$, then
\begin{equation} \label{e:p1}  \int_0^\infty \frac 1 {\mu(t)^{s/d}} \dd t \leq u - \varphi,
\end{equation}
where $\mu(t) := |\{ y : u(y) - u(x) - b \cdot (y-x) < t \}|$.

In particular, taking suppremum in $b$ in the left-hand side of \eqref{e:p1}, we get
\[ \MA u \leq u - \varphi.\]
\end{lemma}

\begin{proof}
Assume the opposite. That is that there exists $b \in \subdiff u(x)$ such that
\[\int_0^\infty \frac 1 {\mu_b(t)^{s/d}} \dd t - u(x) + \varphi(x) = \delta > 0.\]
Without loss of generality, let us also assume that $u(x)=0$ and $b=0$.

For $\eps>0$ small enough, we have
\[\int_\eps^\infty \frac 1 {\mu(t)^{s/d}} \dd t > u - \varphi + \delta/2.\]
The left hand side coincides exactly with $\MA u^\eps$ where
\[ u^\eps (y) = (u(y) - \eps)^+,\]
and is constant in the set $\{ y: u(y) < \eps \}$.

Let us first assume that the contact set $\{u=0\}$ consist in $\{x\}$ only. Then $u^\eps(y) = \eps > u(y)$ in an arbitrarily small neighborhood of $x$. In that set we have $\partial u^\eps = \{0\}$, $\MA u$ is constant and larger than $\delta/2$, and from continuity $\varphi < \delta/4$ if $\eps$ is sufficiently small. Therefore, $\MA u > \delta/4$ where $u^\eps > u$.

Since $u^\eps \geq u$, then $\MA u^\eps \geq \MA u \geq 0$ in the set where $u^\eps = u$ from Lemma \ref{l:monotonicity}.

Therefore, $\MA u^\eps \geq u^\eps - \varphi$ and $u^\eps$ would also be subsolution larger than $\varphi$. Since $u^\eps(y) = u(y) \leq \varphi(y) + w(y)$ for $y$ sufficiently large, then $u^\eps \leq \varphi + w$ from Proposition \ref{p:comparison}

If $\{u=0\}$ consist in more than one point we cannot use exactly the same perturbation since $\MA u - u + \varphi$ may be zero at some points in $\{u=0\}$ other than $x$. Let $x_1$ be the point on the compact set $\{u=0\}$ where $\varphi$ is maximum. Note that at this point we have $\MA u(x_1) \geq \MA u(x)$, $\varphi(x_1) \geq \varphi(x)$ and $u(x)=u(x_1)=0$. In particular, $\MA u(x_1) > u(x_1) - \varphi(x_1) + \delta$.

Since the function $\varphi$ is convex, and $\{u=0\}$ is compact and convex, the point $x_1$ can be chosen to be an extremal point of $\{u=0\}$. This implies that there is a point $x_2$, arbitrarily close to $x_1$, such that $u$ is strictly convex at $x_1$. That is, there is a supporting plane that touches $u$ only at $x_2$.

As a consequence of Lemma \ref{l:c11}, $u \in C^1$. Since $u$ is also convex, we can apply Lemma \ref{l:lower-semicontinuous} and obtain that $\MA u$ is lower semicontinuous. Therefore, we can choose $x_2$ sufficiently close to $x_1$ so that $\MA u(x_2) \geq \MA u(x_1)  + \delta / 4$ and also $\varphi(x_2) < \varphi(x_1) + \delta/4$. In particular $\MA u(x_2) > u(x_2) - \varphi(x_2) + \delta/2$ and we can repeat the argument above to the point $x_2$ instead of $x$ replacing $\delta$ with $\delta/2$.
\end{proof}

Combining the results of Lemmas \ref{l:uissub} and \ref{l:uissuper}, we see that the function $u$ we constructed is indeed the solution we wanted. We state this in the following theorem.

\begin{thm} \label{t:well-posed}
The equation \eqref{e:problem} has a unique solution $u$ such that $u-\varphi \to 0$ at infinity. Moreover, this solution is $C^{1,1}$.
\end{thm}

\section{Some remarks and open questions}

\label{s:questions}

\subsection{The Dirichlet problem in a bounded domain}

It seems tempting to study the following Dirichlet problem.
\begin{align*}
\MA u &= f \qquad \text{ in } B_1, \\
u &= \varphi \qquad \text{ in } \R^d \setminus B_1.
\end{align*}

Here $\varphi$ is a convex function in $\R^d$ and $f$ is any given nonnegative function.

Actually, the problem above does not always have a solution. 

Let $U(x)$ be the convex envelope in $\R^d$ to the restriction of $\varphi$ to $\R^d \setminus B_1$. 
The fact that $u$ must be convex gives us right away that $u$ is less or equal to $U$ in $B_1$.

If $\varphi$ is not an affine function, we will get that $\MA U > 0$ in $B_1$. If we choose $0 \leq f < \MA U$ in $B_1$, then we cannot have a solution $u$ to the problem above because it would contradict the comparison principle from Proposition \ref{p:comparison}.

\subsection{A regularization procedure}

In the proof of Lemma \ref{l:uissub} we implicitly used an approximation of the operator $\MA u$ with one which is easier to study. This approximation may be useful in other contexts, so we state it explicitly in this subsection. The approximated operator is the following.

\begin{equation} \label{e:monsterma-con-piquitos}  \begin{split}
\MA^\eps u(x) = \inf\bigg\{ &\int_{\R^d} (u(x+y) - u(x) - y \cdot \grad u(x)) K(y) \dd y : \\
& \text{ from all kernels $K$ such that } |\{y : K(y)>r^{-n-s}\}| = |B_r|, \\  &\text{ and } K(y) = |y|^{-d-s} \text{ for } |y| < \eps \bigg\}. 
\end{split} 
\end{equation}

The operator $\MA^\eps$ approximates $\MA$ as $\eps \to 0$. The point is that for $\eps>0$, the operator $\MA^\eps u(x)$ is realized by an integral operator $\int (u(x+y) - u(x) - y \cdot \grad u(x)) K(y) \dd y$ with a kernel $K$ such that $K(y) = |y|^{-d-s}$ for $|y| < \eps$. Now, for any bounded right hand side $f$, we have that if
\[ \MA^\eps u = f \text{ in } \Omega,\]
and $f \in L^\infty$, then $u \in C^{1,\alpha}$ in the interior of $\Omega$. This follows from Theorem 6.1 in \cite{caffarelli2011regularity} applied to a suitably rescaled solution $u$.

In Theorem \ref{t:well-posed}, we obtained a solution to our problem \eqref{e:problem} using Perron's method. Another approach may be to construct a solution $u^\eps$ to $\MA^\eps u^\eps = u^\eps -\varphi$, find a uniform $C^{1,1}$ estimate, and pass to the limit as $\eps \to 0$. 

\subsection{Viscosity solutions}

Like for other nonlocal equations (see \cite{barles2008second} and \cite{caffarelli2009regularity}), we can define the concept of viscosity solutions for equations involving the operator $\MA$.

\begin{defn}\label{d:viscosity}
We say that an upper semicontiuous function $u$ satisfies $M_d u \geq f$ (subsolution) in the viscosity sense in a domain $\Omega$, if every time a smooth function $\varphi$ touches $u$ from above at a point $x \in \Omega$ (meaning that $\varphi(x) = u(x)$ whereas $\varphi \geq u$ in a neighborhood of $x$), then the following happens. Let $B_r(x)$ be a neighborhood of $x$ where $\varphi \geq u$. If we construct the auxiliary function $v$ as
\[ 
v(y) = \begin{cases}
u(y) & \text{if } y \notin B_r(x),\\
\varphi(y) & \text{if } y \in B_r(x).
\end{cases}
\]
Then $M_s v(x) \geq f(x)$.

Similarly, we define for lower semicontinuous functions $u$ that $M_du \leq f$ in $\Omega$ (supersolution) in the viscosity saying that for every test function $\varphi$ touching $u$ from below at $x \in \Omega$, the auxiliary function $v$ satisfies $M_s v(x) \leq f(x)$

A viscosity solution of $M_s u =f$ in $\Omega$ is a function which satisfies both $M_s u \geq f$ and $M_s u \leq f$ in $\Omega$ in the viscosity sense.
\end{defn}

In this paper, we did not use the concept of viscosity solutions. We used a pointwise definition of the operator instead explained in section \ref{s:pointwise}. In particular, our comparison result of Proposition \ref{p:comparison} is not proved for viscosity subsolutions and solutions of an equation. The natural question is the following:

\noindent \textbf{Question. } Let $u$ be a convex function. Is it equivalent that $\MA u \geq f$ in the viscosity sense in a domain $\Omega$ with the point-wise definition of $\MA \geq f$ given in section \ref{s:pointwise}?

\subsection{A conjecture about global solutions with constant right hand side}

A classical theorem by Jorgens \cite{jorgens1954losungen}, Calabi \cite{calabi1958improper} and Pogorelov \cite{pogorelov1972improper} says that every global convex solution $u$ to the Monge-Amp\`ere equation
\[ \det D^2 u = 1,\]
must be a quadratic polynomial. We present here an analogous nonlocal statement.

Note that the operator $\MA u$ defined in this paper cannot be applied to functions $u$ which grow quadratically at infinity. This is because the tails of the integrals in Definition \eqref{d:monsterMA} would diverge. We now show a modification of the operator which allows $u$ to grow arbitrarily at infinity.

Given a fixed kernel $K_0$, we define $\tilde \MA u$ by the following formula
\begin{equation} \label{e:monsterma-localized}  \begin{split}
\tilde \MA u(x) = \inf & \bigg\{ \int_{\R^d} (u(x+y) - u(x) - y \cdot \grad u(x)) K(y) \dd y : \\
& \text{ from all kernels $K$ such that for all $t>0$, } |\{y : K(y)>t\}| = |\{y : K_0(y)>t\}| \bigg\}. 
\end{split} 
\end{equation}

Clearly, if $K_0(y) = |y|^{-d-s}$, then $\tilde \MA$ and $\MA$ coincide. In order to make sure that the tails of the integral do not diverge, we select
\[ K_0(y) = |y|^{-d-s} \chi_{B_1}.\]

We emphasize that different choices of $K_0$ may be preferable for different problems.

\noindent \textbf{Question. } Assume $u : \R^d \to \R$ is a convex function which satisfies
\[ \tilde \MA u = 1 \ \text{ in } \R^d.\]
Does it imply that $u(x)$ is a quadratic polynomial?

We believe that the answer to this question should be affirmative, which would provide a nonlocal version of the result by Jorgens, Calabi and Pogorelov.

\subsection{Questions regarding interior regularity}

The $C^{1,1}$ regularity that we obtained for the solutions in Theorem \ref{t:well-posed} is based on global considerations for the problem that we proposed. It is still unclear whether we can prove more flexible local regularity results. For example, a natural question would be the following.

\noindent \textbf{Question. } Assume $u : \R^d \to \R$ is a convex function which satisfies
\[ \MA u = 1 \ \text{ in } B_1.\]
How regular is $u$ in $B_{1/2}$?

A less ambitious but already challenging question is whether the solution $u$ obtained in Theorem \ref{t:well-posed} is more regular than $C^{1,1}$.

\section{Acknowledgments}

L. Caffarelli was partially supported by NSF grant DMS-1160802. L. Silvestre was partially supported by NSF grants DMS-1001629 and DMS-1065979. The authors wish to thank Maik Urban for pointing out some issues in an earlier version of this paper.

\bibliographystyle{plain}   
\bibliography{monsterma}             
\index{Bibliography@\emph{Bibliography}}%

\end{document}